%% file: main.tex
\documentclass[12pt]{ociamthesis}  
\usepackage{amssymb}
\usepackage{amsmath}
\usepackage{titling}
\usepackage{ragged2e}
\usepackage{parskip}
\usepackage{tocloft}
\usepackage{amsthm}
\usepackage[symbol]{footmisc}

\newtheorem{theorem}{Theorem}[section]

\newtheorem*{observation*}{Observation}
\newtheorem*{remark}{Remark}
\newtheorem{corollary}{Corollary}[theorem]
\newtheorem*{theorem*}{Theorem}

\thanksmarkseries{alph}

\begin{document}

\baselineskip=15pt plus1pt

\setcounter{secnumdepth}{4}
\setcounter{tocdepth}{3}

\include{contents.tex/titlepage}       
\addcontentsline{toc}{chapter}{Bibliography}

\renewcommand{\bibname}{References}
\bibliography{refs}      
\bibliographystyle{plain}  
\end{document}

%% file: contents.tex/titlepage.tex
\begin{titlepage}
\vspace{2pt}
\begin{center}
\Large
       \textbf{The extension of Weyl-type relative perturbation bounds to general Hermitian matrices}

\small
\textbf{Haoyuan Ma}\footnote{Email adress: {haoyuan}\_{ma@outlook.com}

\textbf{Funding:} This research did not receive any specific grant from funding agencies in the public, commercial, or not-for-profit sectors.} 
\end{center}
\textit{\small{Mathematical Institute, University of Oxford, Andrew Wiles Building, Woodstock Road, Oxford, United Kingdom, OX2 6GG.}} 

 \textbf{Abstract.}
We extend several relative perturbation bounds to Hermitian matrices that are possibly singular, and also develop a general class of relative bounds for Hermitian matrices. As a result, corresponding relative bounds for singular values of rank-deficient $m\times n$ matrices are also obtained using the Jordan-Wielandt matrices.  We also present that the main relative bound derived would be invariant with respect to congruence transformation under certain conditions, and compare its sharpness with the Weyl's absolute perturbation bound. 

\textbf{Keywords:} Eigenvalue perturbation; Hermitian matrices; Rank-deficient matrices; Relative perturbation bounds; Congruence transformation; Singular value 

\textbf{AMS subject classification:} 15A18, 15A42, 15B57
\section{Introduction}
Given a square matrix $A$, it has been an active research area to give bounds on the perturbed eigenvalues of $A+E$, where $E$ can be interpreted differently, such as any genuine addition to $A$, or a matrix arises due to floating-point errors as considered in [11]. When both $A$ and $E$ are Hermitian, the Weyl's bound [7, Theorem 4.3.1] gives \begin{equation*}
  -\|E\| \leq \lambda_i(A+E)-\lambda_i(A)\leq \|E\|
\end{equation*}
for $i\in \{1,...,n\}$, establishing a one-to-one pairing between $\lambda_i(A+E)$ and $\lambda_i(A)$, and in general a perturbation bound with such correspondence is classified as a Weyl-type bound in [6] and [8]. The Weyl's bound belongs to absolute perturbation results as it concerns the absolute difference between $\lambda_i(A+E)$ and $\lambda_i(A)$, whereas another category of perturbation results is relative bounds that consider the distance between $\lambda_i(A+E)$ and $\lambda_i(A)$ relative to $\lambda_i(A)$. There have been many relative bounds developed in [3, 4, 6, 8, 10, 11] and in particular, the Weyl-type relative bounds always 
 take the form \begin{equation*}
    -k|\lambda_i(A)|\leq \lambda_i(A+E)-\lambda_i(A) \leq k|\lambda_i(A)|
\end{equation*}
for $i\in \{1,...,n\}$ where different constants $k$ correspond to different relative bounds. A condition for these results which is always imposed is that $A$ is Hermitian non-singular
and $E$ is Hermitian, for example, $A$ and $A+E$ are taken to be positive definite in [6, 8, 10]. In contrast, the Weyl's bound  just requires both $A$ and $E$ to be Hermitian. Hence, in this paper our main goal is to extend the existing Weyl-type relative perturbation bounds to the case that $A$ may not have full rank, so to apply these relative bounds we do not need to make stronger assumptions on the structure of $A$ than to apply the Weyl's bound. 

In Section 2, we will first establish a relative bound in Theorem 2.1 which extends [4, Theorem 2.1] considering the case that $A$ has rank $r\leq n$. To follow, we will use this result to give a class of Weyl-type relative perturbation results that not only incorporate the major Weyl-type relative bounds for eigenvalues in [4, 6, 8, 10] but also take into account that $A$ is potentially singular.

The sharpness of the bound in Theorem 2.1 depends on the $i$th eigenvalue of $A$ whereas the Weyl's bound is invariant with respect to the indices, so we will investigate for which indices the relative bound is likely to be sharper than the Weyl's bound in Section 3. Also, we will look into a known relation between Theorem 2.1 and congruence transformation but additionally consider the case that $A$ could be singular.

In Section 4, we will use Theorem 2.1 to give a relative perturbation bound for singular values of any general $m \times n$ matrices of rank $r$.  

\textbf{Preliminaries.} We let $A^*$ denote the Hermitian of $A$. When both $A$ and $B$ are Hermitian, we denote $A\geq \mathbf{0}$ if $A$ is positive semi-definite, and $B\leq A$ if $A-B$ is positive semi-definite.  When $A$ is normal, it has a spectral decomposition $A=VDV^*$ where $D=\mbox{diag}(\lambda_1, ..., \lambda_n)$ and $V$ is unitary. We can always arrange its eigenvalues to be decreasing with respect to the indices which is our default ordering when $A$ is non-singular. When $A$ may not have full rank, the default ordering with respect to the inertia is $\lambda_1\geq...\geq \lambda_r, \lambda_{r+1}=...=\lambda_n=0$ where $r=\mbox{rank}(A)$. 

When $A=VDV^*$ is normal, we define $A^{1/2}=VD^{1/2}V^*=V \mbox{diag}(\lambda_1^{1/2},...,\lambda_n^{1/2})V^*$ to be any normal square root of $A$ as in [4], so that $(A^{1/2})^2=A$. We denote $A=V_rD_rV_r^*$ as its thin spectral decomposition. Since $0^{1/2} \equiv 0$, $A^{1/2}=V_rD_r^{1/2}V_r^*$.

For any matrix $A$, we let $\|A\|$ denote its 2-norm which is $\sigma_1(A)$, where the singular values are arranged in decreasing order. For any square matrix $A$, we have $\|A\|\geq \max\limits_i|\lambda_i(A)|$. In particular, when $A$ is Hermitian, $\|A\|=\max\limits_i|\lambda_i(A)|=|\lambda_{j^*}(A)|$ for some index $j^*$. We let $A^+$ denote its unique pseudo inverse. If $A=U\Sigma V^*=U_r\Sigma_rV_r^*$ is its thin SVD, $A^+=V_r\Sigma_r^{-1}U_r^*$.

\section{Weyl-type relative bounds for eigenvalues}
\subsection{The main Weyl-type relative bound}
The main relative perturbation bound of this section stems from [10, Theorem 2.3] which states that if $A$, $A+E$ are positive definite and $k=\|A^{-1/2}EA^{-1/2}\|$, we get $\lambda_i(A)-k|\lambda_i(A)|\leq \lambda_i(A+E)\leq \lambda_i(A)+k|\lambda_i(A)|$. Theorem 2.1 in [4] extends this result to $A$ being Hermitian non-singular by imposing the additional condition that $k\leq 1$. We are going to show that an analogue also holds without assuming that $A$ is non-singular, by following a similar approach in [4] but considering $A^+$ instead.
 \begin{theorem}
Let $A$ be Hermitian with rank $r\leq n$ and eigenvalues $\lambda_1\geq...\geq \lambda_r, \lambda_{r+1}=...=\lambda_n=0$. Let $E$ be Hermitian and the eigenvalues of $A+E$ be in decreasing order.   If $k=\|(A^{1/2})^{+} E(A^{1/2})^{+}\|\leq 1$, we have for any $i\in \{1,...,r\},$
\begin{equation}
\lambda_{n-r+i}(A+E)\leq \lambda_i+k|\lambda_i|
    \end{equation}
and \begin{equation}
    \lambda_i(A+E) \geq \lambda_i -k|\lambda_i|
\end{equation}
   
\begin{proof}
 Take the spectral decomposition $A=VDV^*=V_rD_rV_r^*$ and hence its polar factorisation $A=PS=SP$ by taking  \begin{equation*}
     P=\left[\begin{array}{cc}
        V_r  &  V_r^\perp
         \end{array}\right] \left[\begin{array}{cc}
     |D_r| &  \\
      & \mathbf{0}
 \end{array}\right] \left[\begin{array}{cc}
      V_r^*  \\
      {V_r^{\perp}}^*
 \end{array}\right]
\end{equation*}
\begin{equation*}
S=\left[\begin{array}{cc}
    V_r &  V_r^\perp
\end{array}\right] \left[\begin{array}{cc}
     H_r &  \\
      & I_{n-r} 
 \end{array}\right] \left[\begin{array}{cc}
      V_r^*  \\
     {V_r^{\perp}}^*
 \end{array}\right]
 \end{equation*}
where we denote $H_r=\mbox{sign}(D_r)$. As a result, \begin{equation*} \begin{array}{cc}
(A^{1/2})^+=V_rD_r^{-1/2}V_r^*, & (P^{1/2})^+=V_r|D_r|^{-1/2}V_r^* \end{array}
\end{equation*} Because $
    \lambda_j^{1/2}=(\mbox{sign}(\lambda_j))^{1/2}|\lambda_j|^{1/2}$ for all $1 \leq j \leq r$,    $|\lambda_j|^{-1/2}=(\mbox{sign}(\lambda_j))^{1/2}\lambda_j^{-1/2}$ where $|\lambda_j|^{-1/2}$ are entries of $|D_r|^{-1/2}$, $(\mbox{sign}(\lambda_j))^{1/2}$ are entries of $H_r^{1/2}$ and $\lambda_j^{-1/2}$ are entries of $D_r^{-1/2}$. Then, $
|D_r|^{-1/2}=D_r^{-1/2}H_r^{1/2}=H_r^{1/2}D_r^{-1/2}$. 
Therefore, \begin{equation} \begin{aligned}
    (P^{1/2})^{+}
    &=(V_rD_r^{-1/2}V_r^*)(V_rH_r^{1/2}V_r^*)\\
    &=V_rD_r^{-1/2}V_r^*\left[\begin{array}{cc}
         V_r & V_r^\perp 
    \end{array}\right] \left[\begin{array}{cc}
     H_r^{1/2}&  \\
      & I_{n-r} 
 \end{array}\right]\left[\begin{array}{cc}
      V_r^*  \\
      {V_r^{\perp}}^*
 \end{array}\right]\\
    &=(A^{1/2})^{+}W=W(A^{1/2})^+
        \end{aligned}
    \end{equation} 
  where the last equality in (3) is obtained by re-doing the derivation while swapping $V_rD_r^{-1/2}V_r^*$ and $V_rH_r^{1/2}V_r^*$ around. $W$ is unitarily diagonalisable with eigenvalues being roots of unity as shown in (3), so $W$ is unitary, which leads to
    \begin{equation}
        \begin{aligned}
        k&=\|(A^{1/2})^{+}E(A^{1/2})^{+}\|\\
        &= \|W(A^{1/2})^{+}E(A^{1/2})^{+}W\|\\
        &=\|(P^{1/2})^{+}E(P^{1/2})^{+}\|
        \end{aligned}
        \end{equation}
Since both $(P^{1/2})^+$ and $E$ are Hermitian, so is $(P^{1/2})^{+}E(P^{1/2})^{+}$. Following (4), we have $-kI_n \leq (P^{1/2})^{+}E(P^{1/2})^{+}\leq kI_n$. Multiplying all three sides both left and right by $P^{1/2}$ preserves the order, and $(P^{1/2})^+P^{1/2}=P^{1/2}(P^{1/2})^+=V_rV_r^*$, so we get $-kP\leq V_rV_r^*EV_rV_r^*\leq kP$. Adding $A$ to all sides gives \begin{equation*} 
 A-kP \leq A+V_rV_r^*EV_rV_r^* \leq A+kP
 \end{equation*} 
 Noticing that $V_rV_r^*AV_rV_r^*=V_rV_r^*V_rD_rV_rV_r^*V_r^*=V_rD_rV_r^*=A$, we have \begin{equation}
    A-kP\leq V_rV_r^*(A+E)V_rV_r^*\leq A+kP
\end{equation}
Applying the monotonicity theorem [7, Corollary 4.3.12] to (5), we get \begin{equation}
    \lambda_i (A-kP)\leq \lambda_i (V_rV_r^*(A+E)V_rV_r^*) \leq \lambda_i (A+kP)
\end{equation}

To find the correspondence between  $\lambda_i(V_rV_r^*(A+E)V_rV_r^*)$ and eigenvalues of $A+E$, we apply that $\mbox{eig}(XY)=\mbox{eig}(YX)$ [7, Theorem 1.3.22]: for $i\in \{1,...,r\}$, \begin{equation*} \begin{aligned}
    \lambda_i(\underbrace{V_rV_r^*}_X\underbrace{(A+E)V_rV_r^*}_Y)&=\lambda_i((A+E)(V_rV_r^*V_rV_r^*))\\
    &=\lambda_i((A+E)V_rV_r^*)\\
    &=\lambda_i(V_r^*(A+E)V_r)
\end{aligned}
\end{equation*}
To consider $V_r^*(A+E)V_r$ in relation to $A+E$, we let \begin{equation} \begin{aligned}
    V^*(A+E)V&= \left[\begin{array}{cc}
       V_r^*  \\
       {V_r^{\perp}}^*
    \end{array}\right](A+E) \left[\begin{array}{cc}
        V_r & V_r^\perp
        \end{array}\right] \\
    &=\left[\begin{array}{cc}
       V_r^*(A+E)V_r  & V_r^*(A+E)V_r^\perp \\
        {V_r^{\perp}}^*(A+E)V_r & {V_r^{\perp}}^*(A+E)V_r^\perp
    \end{array}\right]
\end{aligned}
\end{equation}
Hence, $V_r^*(A+E)V_r$ is the $r\times r$ principal sub-matrix of $V^*(A+E)V$, both of which are Hermitian. Using Theorem 4.3.28 in [7], we get for $i\in \{1,...,r\}$, \begin{equation}\lambda_{n-r+i}(V^*(A+E)V)\leq \lambda_i(V_r^*(A+E)V_r)\leq \lambda_i(V^*(A+E)V) \end{equation} $V^*(A+E)V$ and $A+E$ have identical eigenvalues, so (8) becomes $\lambda_{n-r+i}(A+E)\leq \lambda_i(V_r^*(A+E)V_r)\leq \lambda_i(A+E)$. Hence, for $i\in\{1,...,r\}$, \begin{equation} \begin{aligned}
     \lambda_{n-r+i}(A+E)&\leq \lambda_i(V_r^*(A+E)V_r)\\
     &= \lambda_i(V_rV_r^*(A+E)V_rV_r^*)\\
     &\overset{(6)}{\leq} \lambda_i(A+kP)
     \end{aligned} 
\end{equation}
and \begin{equation} \begin{aligned}
    \lambda_i(A+E)&\geq \lambda_i(V_r^*(A+E)V_r)\\
    &=\lambda_i(V_rV_r^*(A+E)V_rV_r^*)\\
    &\overset{(6)}{\geq} \lambda_i(A-kP)
\end{aligned}
    \end{equation}
 where $A \pm kP= V_r (D_r \pm k|D_r|) V_r^*$.  $k\leq 1$ is sufficient to ensure that $\lambda_j \leq \lambda_{j'}$ always implies that $\lambda_j \pm k|\lambda_j|\leq \lambda_{j'}\pm k|\lambda_{j'}| $, because the functions $f, g:\mathbb{R} \rightarrow \mathbb{R}$ defined by $f(x)=x + k|x|, g(x)=x-k|x|$ are increasing over $\mathbb{R} $ when $k\leq 1$. So for any $i$, \begin{equation}
 \lambda_i (A \pm kP)= \lambda_i \pm k|\lambda_i| \end{equation}
 Substituting (11) into (9) and (10) we get (1) and (2) for $i\in \{1,...,r\}$ respectively. 
 
\end{proof}
\end{theorem}
\begin{remark} \textnormal{As pointed out in [4], the condition $k\leq 1$ is necessary in order to have $\lambda_i (A \pm kP)= \lambda_i \pm k|\lambda_i|$ because when $k>1$, for $ |\lambda_j|\gg|\lambda_{j'}|$ while $\lambda_j<0$, $\lambda_{j'}>0$, $\lambda_j+k|\lambda_j|>\lambda_{j'}+k|\lambda_{j'}|$ so (11) is not true in this case. Even when $A$ is positive semi-definite so $\lambda_l\geq 0$ for all $l$, (11) does not hold for $k>1$, because the function $g(x)=x-k|x|$ is decreasing on $[0, \infty)$ when $k>1$. However, following the argument in [4], we can drop the constraint that $k\leq 1$ when both $A, A+E$ are positive semi-definite, recovering an analogue of [6, Corollary 4.2] and [10, Theorem 2.3]:}
\end{remark}

\begin{corollary}
Let $A$ be positive semi-definite with rank $r\leq n$ and eigenvalues $\lambda_1\geq...\geq \lambda_r>0=\lambda_{r+1}=...=\lambda_n$. Let $A+E$ also be  positive semi-definite with eigenvalues in decreasing order, $k=\|(A^{1/2})^+E(A^{1/2})^+\|$. Then for any $i\in \{1,...,r\}$, 
\begin{equation*} 
\begin{array}{cc}
\lambda_{n-r+i}(A+E)\leq (1+k)\lambda_i, & \lambda_i(A+E) \geq (1-k)\lambda_i\\
\end{array}
\end{equation*}
   \end{corollary}
    \begin{proof}
       We follow the same proof as Theorem 2.1 until getting (9) and (10), where $A \pm kP= V_r (D_r \pm k|D_r|) V_r^*$. Since $A$ is positive semi-definite, $\lambda_l\geq 0$ for all $l$. As a result, for any $k$, $\lambda_j\leq \lambda_{j'}$ implies that $\lambda_j +k\lambda_j\leq \lambda_{j'}+k\lambda_{j'}$, which means $\lambda_i(A+kP)=\lambda_i+k\lambda_i$ for any $i$. Substituting it into (9), we get (1)  for $i\in \{1,...,r\}$. 
       
       On the other hand, for any $i\in \{1,...,r\}$, when $k>1$, $\lambda_i-k\lambda_i\leq 0$, while $\lambda_i(A+E)\geq 0$ because $A+E$ is positive semi-definite, so $\lambda_i-k\lambda_i\leq \lambda_i(A+E).$ We have shown that this also holds when $k\leq 1$ in Theorem 2.1, so (2) always holds in this case.    
       
       \end{proof} 
       
$\|(A^{1/2})^+E(A^{1/2})^+\|$ is not the only choice of $k$ to give a relative bound. As shown in the following result which uses a similar proof as [4, Lemma 2.2], $\|A^+E\|$ or $\|EA^{+}\|$ can also take the role of $\|(A^{1/2})^+E(A^{1/2})^+\|$ in Theorem 2.1. 
\begin{theorem}
Let $A$ be  normal with rank $r\leq n$ and $E$ be Hermitian, 
then 
\begin{equation*}
    \|(A^{1/2})^+E(A^{1/2})^+\|\leq \|A^+E\|=\|EA^+\|
    \end{equation*}
 \begin{proof}
    Take the respective polar decomposition $A=PS=SP$ and $E=\Tilde{P}\Tilde{S}=\Tilde{S}\Tilde{P}$ where $P=V_r|D_r|V_r^*.$ Then $A^+=P^+S^*=S^*P^+$, where $A^+$, $P^+$ are Hermitian, so
    \begin{equation}
\|A^+E\|=\|S^*P^+\Tilde{P}\Tilde{S}\|=\|P^+\Tilde{P}\|=\|\Tilde{P}^*(P^+)^*\|=\|\Tilde{S}^*\Tilde{P}^*(P^+)^*S^*\|=\|EA^+\|
        \end{equation}
Using (4), $\|(A^{1/2})^+E(A^{1/2})^+\|=\|(P^{1/2})^+E(P^{1/2})^+\|.$
As $(P^{1/2})^+=V_r|D_r|^{-1/2}V_r^*$ is Hermitian, $(P^{1/2})^+E(P^{1/2})^+$ is Hermitian. Therefore, using $\mbox{eig}(XY)=\mbox{eig}(YX)$, \begin{equation}
\begin{aligned}
     \|(P^{1/2})^+E(P^{1/2})^+\|&=\max\limits_i|\lambda_i(\underbrace{(P^{1/2})^+}_{X}\underbrace{E(P^{1/2})^+}_{Y})|\\
     &=\max\limits_i|\lambda_i(E(P^{1/2})^+(P^{1/2})^+)|\\
     &=\max\limits_i|\lambda_i(EP^+)|\\
     &\leq \|EP^+\|=\|E(P^+S^*)\|=\|EA^+\|
     \end{aligned}
\end{equation}
\end{proof}
\end{theorem}
\subsection{A class of Weyl-type relative bounds}
Having seen relative bounds of the form (1) and (2) where the choices of $k$ can be $\|(A^{1/2})^+E(A^{1/2})^+\|$, $\|A^+E\|$ and $\|EA^+\|$, a follow-up question is that whether they belong to a larger class of relative perturbation results. The motivating result of this mainly comes from [6, Corollary 2.4] giving a general form of Bauer-Fike type relative bounds, which we will state below:
\begin{theorem}[Eisenstat and Ipsen, 1998]
    Let $A$ be a non-singular diagonalisable matrix. If $A=A_1A_2=A_2A_1$ and $\kappa (X)$ is the conditional number of eigenvector matrix of $A$, then for any given eigenvalue $\lambda (A+E)$, \begin{equation*}
        \min\limits_i\frac{|\lambda_i(A)-\lambda(A+E)|}{|\lambda_i(A)|}\leq \kappa (X)\|A_1^{-1}EA_2^{-1}\|
    \end{equation*}
    \end{theorem}
    It was proven in [6] by applying the Bauer-Fike theorem in [1] to $A_2AA_2^{-1}$ and $A_2EA_2^{-1}$. In order to find a class of relative bounds incorporating the choices of $k$ in Theorem 2.1 and 2.2, we establish a class of relative bounds for $A, E$ being Hermitian with a one-to-one pairing between $\lambda_i(A+E)$ and $\lambda_i(A)$ in the following theorem, which will be used later to give an analogue of Theorem 2.3.
 \begin{theorem}
Let $A$ be Hermitian of rank $r\leq n$ with eigenvalues $\lambda_1\geq...\geq \lambda_r, \lambda_{r+1}=...=\lambda_n=0$, and polar factor $P$. Let $E$ be Hermitian, $A_1=P_1S_1$ be its left polar factorisation, $A_2=S_2P_2$ be its right polar factorisation. If $P=P_1P_2=P_2P_1$ and $k=\|A_1^+EA_2^+\|\leq 1$, then for any $i\in \{1,...,r\},$
\begin{equation}
\begin{array}{cc}
\lambda_{n-r+i}(A+E)\leq \lambda_i+k|\lambda_i|, & \lambda_i(A+E) \geq \lambda_i -k|\lambda_i|
\end{array}\end{equation}        
where the eigenvalues of $A+E$ are in decreasing order.
\end{theorem}
\begin{proof}
  Let $A=VDV^*=V_rD_rV_r^*$ be its spectral decomposition with the eigenvalues arranged in the specified ordering, then $P=V_r|D_r|V_r^*$.  Since $PP_1=P_1P_2P_1=P_1P,$ $ PP_2=P_2P_1P_2=P_2P$
 and $P_1P_2=P_2P_1$, we have $P, P_1$ and $P_2$ forming a commuting family. Since they are all Hermitian, by Theorem 2.5.5 in [7], $P_1$ and $P_2$ are also diagonalisable through $V$. 
Let \begin{equation*}
    P_1=V\Lambda V^*=\left[\begin{array}{cc}
        V_r  &  V_r^\perp
         \end{array}\right] \left[\begin{array}{cc}
     \Lambda_r &  \\
      & \Lambda_{n-r}
 \end{array}\right] \left[\begin{array}{cc}
      V_r^*  \\
      {V_r^{\perp}}^*
 \end{array}\right]\end{equation*}
 \begin{equation*}
      P_2=V\Sigma V^*=\left[\begin{array}{cc}
        V_r  &  V_r^\perp
         \end{array}\right] \left[\begin{array}{cc}
     \Sigma_r &  \\
      & \Sigma_{n-r}
 \end{array}\right] \left[\begin{array}{cc}
      V_r^*  \\
     {V_r^{\perp}}^*
 \end{array}\right] \end{equation*}
 Since $P=P_1P_2$, we have $V|D|V^*=V\Lambda\Sigma V^*$. Multiplying its both sides by $V^*$ on the left and $V$ on the right, we have \begin{equation}
     \left[\begin{array}{cc}
     |D_r| &  \\
      & \mathbf{0}
 \end{array}\right] = \left[\begin{array}{cc}
     \Lambda_r\Sigma_r &  \\
      & \Lambda_{n-r}\Sigma_{n-r}
 \end{array}\right] 
 \end{equation}
where both $\Lambda_r$ and $\Sigma_r$ are invertible because all diagonal entries of $|D_r|$ are non-zero. Hence, $P_1^+=V\left[\begin{array}{cc}
     \Lambda_r^{-1}  &  \\
     & \Lambda_{n-r}^+\
  \end{array}\right]V^*$ and $P_2^+=V\left[\begin{array}{cc}
     \Sigma_r^{-1}  &  \\
     & \Sigma_{n-r}^+
  \end{array}\right]V^*$. In addition, since $A_1^+=S_1^*P_1^+$ and $A_2^+=P_2^+S_2^*$, \begin{equation*}
k=\|A_1^+EA_2^+\|=\|S_1^*P_1^+EP_2^+S_2^*\|=\|P_1^+EP_2^+\|
 \end{equation*}
Our next aim is to show that \begin{equation}
     A-kP\leq A+V_rV_r^*EV_rV_r^*\leq A+kP
      \end{equation}     
      because if this holds,  using the same reasoning starting from (5) in Theorem 2.1, we can obtain (14). Let $P_1^r=V_r\Lambda_rV_r^*$ and $P_2^r=V_r\Sigma_r V_r^*$, then
      \begin{equation*}
\begin{aligned}
    P_1^rP_1^+&=V\left[\begin{array}{cc}
         \Lambda_r &  \\
          & \mathbf{0}
     \end{array}\right]V^*V\left[\begin{array}{cc}
        \Lambda_r^{-1} &  \\
         & \Lambda_{n-r}^+
    \end{array}\right]V^*=V_rV_r^*
    \end{aligned}
 \end{equation*}
  Similarly, $P_2^+P_2^r=V\left[\begin{array}{cc}
     \Sigma_r^{-1}  &  \\
       & \Sigma_{n-r}^+
  \end{array}\right]\left[\begin{array}{cc}
     \Sigma_r  &  \\
       & \mathbf{0}
  \end{array}\right]V^*=V_rV_r^*$. Therefore, we have \begin{equation}
P_1^rP_1^+EP_2^+P_2^r=V_rV_r^*EV_rV_r^*
  \end{equation} When $k=\|P_1^+EP_2^+\|=0$, $P_1^+EP_2^+=\mathbf{0}$ and  $V_rV_r^*EV_rV_r^*=\mathbf{0}$ by (17), so (16) always holds. We can assume $k>0$ from now on.  
  To show that $ P_1^rP_1^+EP_2^+P_2^r\leq kP$, it is equivalent to show that the eigenvalues of $kP-P_1^rP_1^+EP_2^+P_2^r=V_rk|D_r|V_r^*-V_rV_r^*EV_rV_r^*=V_r(k|D_r|-V_r^*EV_r)V_r^*$ are non-negative, which are all the $r$ eigenvalues of $k|D_r|-V_r^*EV_r$ together with $n-r$ copies of zeros. 
 So it suffices to show that all eigenvalues of $k|D_r|-V_r^*EV_r$ are non-negative. Let $\Tilde{A}=k|D_r|$ and  $\Tilde{E}=-V_r^*EV_r$. Using (2) in Theorem 2.1, we get for all $i\in\{1,...,r\}$, \begin{equation}
     \lambda_i(\Tilde{A}+\Tilde{E} )\geq \lambda_i(\Tilde{A})-\|\Tilde{A}^{-1/2}\Tilde{E}\Tilde{A}^{-1/2}\|\lambda_i(\Tilde{A})
  \end{equation}
Since $V_r^*P_1^rP_1^+EP_2^+P_2^rV_r\overset{(17)}{=}V_r^*(V_rV_r^*EV_rV_r^*)V_r=-\Tilde{E}$ and $|D_r|^{-1/2}$ are Hermitian,  so is $|D_r|^{-1/2}V_r^*P_1^rP_1^+EP_2^+P_2^rV_r|D_r|^{-1/2}$. Using $\mbox{eig}(XY)=\mbox{eig}(YX)$, \begin{equation}
  \begin{aligned}
      \|\Tilde{A}^{-1/2}\Tilde{E}\Tilde{A}^{-1/2}\| 
      &=k^{-1}\||D_r|^{-\frac{1}{2}}V_r^*P_1^rP_1^+EP_2^+P_2^rV_r|D_r|^{-\frac{1}{2}}\|\\
      &=k^{-1}\max\limits_j|\lambda_j(\underbrace{|D_r|^{-\frac{1}{2}}V_r^*P_1^r}_{X}\underbrace{P_1^+EP_2^+P_2^rV_r|D_r|^{-\frac{1}{2}}}_{Y})|  \\ 
&=k^{-1}\max\limits_j|\lambda_j(P_1^+EP_2^+P_2^rV_r|D_r|^{-1}V_r^*P_1^r)|\\
&=k^{-1}\max\limits_j|\lambda_j(P_1^+EP_2^+P_2^rP^+P_1^r)|\\
\end{aligned}
    \end{equation}

 Recall (15), $|D_r|^{-1}=(\Lambda_r \Sigma_r)^{-1}=\Sigma_r^{-1}\Lambda_r^{-1}$,then 
  \begin{equation}
\begin{aligned}P_2^rP^+P_1^r&=V_r\Sigma_r V_r^*V_r|D_r|^{-1}V_r^*V_r\Lambda_r V_r^*\\
&=V_r\Sigma_r(\Sigma_r^{-1}\Lambda_r^{-1})\Lambda_rV_r^*\\
          &=V_rV_r^*
          \end{aligned}
      \end{equation}  
Hence, substituting (20) into (19), we have 
\begin{equation}
\begin{aligned}
\|\Tilde{A}^{-1/2}\Tilde{E}\Tilde{A}^{-1/2}\|&=k^{-1}\max\limits_j|\lambda_j(P_1^+EP_2^+V_rV_r^*)|\\
&\leq k^{-1}\|P_1^+EP_2^+V_rV_r^*\|\\&\leq k^{-1}\|P_1^+EP_2^+\|\|V_rV_r^*\|=k^{-1}k=1
\end{aligned}
\end{equation}
Combining (21) with (18) we have $  \lambda_i(\Tilde{A}+\Tilde{E} )\geq \lambda_i(\Tilde{A})-\lambda_i(\Tilde{A})=0$ for any $i\in \{1,...,r\}$. Therefore,  all $r$ eigenvalues of $k|D_r|-V_r^*EV_r$ are non-negative which shows that $kP-P_1^rP_1^+EP_2^+P_2^r$ is positive semi-definite, so $V_rV_r^*EV_rV_r^*=P_1^rP_1^+EP_2^+P_2^r \leq kP$. 

Similarly, $P_1^rP_1^+EP_2^+P_2^r \geq -kP$ is obtained by showing that $\lambda_i(k|D_r|+V_r^*EV_r)\geq 0$ for all $i\in \{1,...,r\}$, which is a result of repeating the derivation from (18) to (21) but for $\Tilde{A}=k|D_r|$ and $\Tilde{E}=V_r^*EV_r.$ \footnote[2]{The change in sign of $\Tilde{E}$ does not affect the related 2-norm or absolute value.} Hence, we have $ -kP\leq V_rV_r^*EV_rV_r^*\leq kP$, so (16) holds. If we now repeat the argument starting from (5) to the end of Theorem 2.1, we will attain (14) for any $i\in\{1,..,r\}$.

\end{proof}
 We can now give a class of relative bounds for Hermitian matrices that resembles Theorem 2.3 more:

\begin{theorem}
    Let $A$ be Hermitian with rank $r\leq n$ and eigenvalues $\lambda_1\geq...\geq \lambda_r, \lambda_{r+1}=...=\lambda_n=0$, and let $E$ be Hermitian. If $A_1, A_2$ are matrices such that  $A=A_1A_2=A_2A_1$, and $k=\|A_1^+EA_2^+\|\leq1 $, then for any $i\in \{1,...,r\}$,\begin{equation*}
\begin{array}{cc}
\lambda_{n-r+i}(A+E)\leq \lambda_i+k|\lambda_i|, & \lambda_i(A+E) \geq \lambda_i -k|\lambda_i|
\end{array}\end{equation*}
 where the eigenvalues of $A+E$  are in decreasing order.\end{theorem}
 \begin{proof}
Since $A=A_1A_2=A_2A_1$ and $A=A^*$,  \begin{equation*}
    AA_1A_1^*=A_1A_2A_1A_1^*=A_1A_1^*A_2^*A_1^*=A_1A_1^*A
\end{equation*}
\begin{equation*}
    AA_2^*A_2=A_2^*A_1^*A_2^*A_2=A_2^*A_2A_1A_2=A_2^*A_2A
\end{equation*}
\begin{equation*}
A_1A_1^*A_2^*A_2=A_1A_2A_1A_2=A_2^*A_1^*A_2^*A_1^*=A_2^*A_2A_1A_1^*
\end{equation*}
As a result, $A, A_1A_1^*$ and $A_2^*A_2$ form a commuting family of Hermitian matrices. By [7, Theorem 2.5.5], there is a unitary matrix $V$ such that $A=VDV^*$ where the eigenvalues are in the specified order above, and also $A_1A_1^*=V\Lambda_1V^*$ and $A_2^*A_2=V\Lambda_2V^*$ where $\Lambda_1, \Lambda_2$ are diagonal, so $A_1A_1^*A_2^*A_2=V\Lambda_1V^*V\Lambda_2V^*=V\Lambda_1\Lambda_2V^*=V\Lambda_2\Lambda_1V^*$. Since $A_1A_1^*, A_2^*A_2$ are positive semi-definite, there exist non-negative diagonal $\Sigma_1, \Sigma_2$ such that $\Lambda_1=\Sigma_1^2$ and $\Lambda_2=\Sigma_2^2$. Because $A^2=VD^2V^*$ and also \begin{equation*}
  A^2=(A_1A_2)(A_1A_2)=A_1(A_2A_1)A_2=A_1A_1^*A_2^*A_2=V\Sigma_1^2\Sigma_2^2V^*=V\Sigma_2^2\Sigma_1^2V^*
\end{equation*}
 we have $D^2=\Sigma_1^2\Sigma_2^2=\Sigma_2^2\Sigma_1^2$, where $D$ has all real entries since $A$ is Hermitian, while $\Sigma_1, \Sigma_2$ have non-negative entries. Therefore, $|D|=\Sigma_1\Sigma_2=\Sigma_2\Sigma_1$, so the polar factor of $A$ is $P=V|D|V^*=V\Sigma_1\Sigma_2V^*=V\Sigma_2\Sigma_1V^*$. 

On the other hand, by Theorem 7.3.1 in [7], the left polar factor of $A_1=P_1S_1$ is $P_1=(A_1A_1^*)^{1/2}=(V\Sigma_1^2V^*)^{1/2}=V\Sigma_1V^*$, and the right polar factor of $A_2=S_2P_2$ is $P_2=(A_2^*A_2)^{1/2}=(V\Sigma_2^2V^*)^{1/2}=V\Sigma_2V^*.$ Therefore, we see $P=P_1P_2=P_2P_1$, so we can apply Theorem 2.4 and obtain (14) for $i\in \{1,...,r\}$. \end{proof}
\begin{remark}
\textnormal{When $A$ and $A+E$ are both positive semi-definite, the condition $k\leq 1$ can be dropped in Theorem 2.4 and Theorem 2.5 because of the same reasoning as in Corollary 2.1.1 apart from the different definition of $k$. Theorem 2.5 includes all the previous choices of $k$ in Theorem 2.1 and Theorem 2.2 by taking $A_1=A_2=A^{1/2}$, and $A_1=A$, $A_2=I$ or $A_1=I$, $A_2=A$.}
\end{remark}

\section{Property and sharpness of the main result}
\subsection{Relation to congruence transformation}
[8, Corollary 2.2] and [10, Theorem 2.4] show that $k=\|A^{-1/2}EA^{-1/2}\|$ exhibits the property of being invariant when $A$ and $E$ are both positive definite and congruence transformed by the same invertible matrix $D$. Such $D$ can be taken to be diagonal to extract the scaling out of $A$ so that $DAD^*$ has entries and eigenvalues about the same order of magnitude while those of $A$ may be widely varying, as already exploited in [3] and [11].  In [6, Corollary 3.4], this is extended to the case that $A$ is non-singular while its polar factor and $E$ are both congruence transformed by $D$. We will show that $A$ being singular can also be taken into account by imposing one extra constraint on $D$, which vanishes when $A$ has full rank so [6, Corollary 3.4] is recovered.
\begin{theorem}
  Let $A$ be normal with rank $r\leq n$, an eigenvector matrix $V$ and polar factor $P$. If $\tilde{A}$ is normal with polar factor $\tilde{P}=DPD^*$, and $\tilde{E}=DED^*$ where $D$ is invertible such that $D^*D$ commute with $V_rV_r^*$, then \begin{equation}k=\|(A^{1/2})^+E(A^{1/2})^+\|=\|(\tilde{A}^{1/2})^+ \tilde{E}(\tilde{A}^{1/2})^+\|
       \end{equation}
    \end{theorem} 
    \begin{proof}
     Since $\Tilde{P}^{1/2}, P^{1/2}$ are Hermitian, $\Tilde{P}=\Tilde{P}^{1/2}\Tilde{P}^{1/2}=\Tilde{P}^{1/2}(\Tilde{P}^{1/2})^*$ and \begin{equation}
\Tilde{P}=DPD^*=DP^{1/2}P^{1/2}D^*=DP^{1/2}(P^{1/2})^*D^*=DP^{1/2}(DP^{1/2})^*
        \end{equation}
        $\Tilde{P}^{1/2}(\Tilde{P}^{1/2})^*$ and $DP^{1/2}(DP^{1/2})^*$ have the same spectral decomposition $U\Lambda U^*$ with the eigenvalues in decreasing order.  Then the SVD of $\Tilde{P}^{1/2}$ is $\Tilde{P}^{1/2}=U\Sigma Q_1^*$ where $Q_1$ is unitary, and the SVD of $DP^{1/2}$ is $DP^{1/2}=U\Sigma Q_2^*$ where $Q_2$ is unitary. Hence, $\Tilde{P}^{1/2}=(DP^{1/2})W$ where $W=Q_2Q_1^*$ is unitary. Taking pseudo inverse both sides, we get $(\Tilde{P}^{1/2})^+=((DP^{1/2})W)^+=W^*(DP^{1/2})^+$. 

Consider $(DP^{1/2})^+$. If $A=V_r\hat{\Lambda}_rV_r^*$ is its thin spectral decomposition, $P=V_r|\hat{\Lambda}_r|V_r^*$ and $P^+=V_r|\hat{\Lambda}_r|^{-1}V_r^*$ so we have $P^{1/2}(P^{1/2})^+=V_r|\hat{\Lambda}_r|^{1/2}V_r^*V_r|\hat{\Lambda}_r|^{-1/2}V_r^*=V_rV_r^*$. As a result, 
 \begin{equation}
 \begin{aligned}
    (DP^{1/2})(P^{1/2})^+D^{-1}=DV_rV_r^*D^{-1}
    \end{aligned}
    \end{equation}Since $D^*DV_rV_r^*=V_rV_r^*D^*D$ by the condition given, multiplying its both sides by $D^{-*}$ on the left and $D^{-1}$ on the right, we get \begin{equation*} 
\begin{aligned}
    DV_rV_r^*D^{-1}=D^{-*}V_rV_r^*D^*
    =(DV_rV_r^*D^{-1})^*
     \end{aligned}
\end{equation*}
Hence, $(DP^{1/2})(P^{1/2})^+D^{-1}$ is Hermitian by (24). On the other hand, 
\begin{equation*}
    (P^{1/2})^+D^{-1}(DP^{1/2})(P^{1/2})^+D^{-1}=(P^{1/2})^+P^{1/2}(P^{1/2})^+D^{-1}=(P^{1/2})^+D^{-1}
\end{equation*}
 \begin{equation*}  
   DP^{1/2}((P^{1/2})^+D^{-1})DP^{1/2}=DP^{1/2}(P^{1/2})^+P^{1/2}=DP^{1/2}
    \end{equation*}
    \begin{equation*}
    (P^{1/2})^+D^{-1}(DP^{1/2})=(P^{1/2})^+P^{1/2}=V_r|\hat{\Lambda}_r|^{-1/2}|\hat{\Lambda}_r|^{1/2}V_r^*=V_rV_r^*
\end{equation*}
Therefore, $(DP^{1/2})^+=(P^{1/2})^+D^{-1}$. Since $(\Tilde{P}^{1/2})^+=W^*(P^{1/2})^+D^{-1}$ and $(\Tilde{P}^{1/2})^+$, $(P^{1/2})^+$ are Hermitian,  $(\Tilde{P}^{1/2})^+=(W^*(P^{1/2})^+D^{-1})^*=D^{-*}(P^{1/2})^+W$. Then \begin{equation}
    \begin{aligned}
        k=\|(\Tilde{A}^{1/2})^+\Tilde{E}(\Tilde{A}^{1/2})^+\|&\overset{(4)}{=}\|(\tilde{P}^{1/2})^+\Tilde{E}(\tilde{P}^{1/2})^+\|\\
        &=\|(W^*(P^{1/2})^+D^{-1})(DED^*)(D^{-*}(P^{1/2})^+W)\|\\
        &=\|W^*(P^{1/2})^+E(P^{1/2})^+W\|\\
        &=\|(P^{1/2})^+E(P^{1/2})^+\|\overset{(4)}{=}\|(A^{1/2})^+E(A^{1/2})^+\|
    \end{aligned}
\end{equation}
\end{proof}
\begin{remark}
\textnormal{As shown in the proof of Theorem 3.1, the condition that $D^*D$ commutes with $V_rV_r^*$ is sufficient to imply $(DP^{1/2})^+=(P^{1/2})^+D^{-1}$, which means the pseudo inverse of $DP^{1/2}$ follows the reverse order law. More generally for any matrices $X, Y$, $(XY)^+=Y^+X^+$ if and only if $X^+X$ commutes with $YY^*$ and $YY^+$ commutes with $X^*X$, as suggested by [2, Theorem 3.1] and [9, Corollary 3.11].}
\end{remark}
We can combine Theorem 3.1 with Theorem 2.1 to extend [8, Corollary 2.2] from the case $A$ being positive definite to $A$ being positive semi-definite and suggest the relative bounds of (1) and (2) for $A+E$ are the same as its congruence transformation $D(A+E)D^*$ under certain conditions on $D$. \begin{corollary}
    Let $\Tilde{A}=DAD^*$ be positive semi-definite with rank $s\leq n$ and eigenvalues $\lambda_1\geq...\geq \lambda_s, \lambda_{s+1}=...=\lambda_n=0$. Let $\Tilde{A}+\Tilde{E}$ be also positive semi-definite with eigenvalues in decreasing order, where $\Tilde{E}=DED^*$.  If $k=\|(A^{1/2})^{+} E(A^{1/2})^{+}\|$, and $D$ is invertible with $D^*DV_rV_r^*=V_rV_r^*D^*D$ where $V$ is an eigenvector matrix of $A$ which is of rank $r$, we have for any $i\in \{1,...,s\},$
\begin{equation*}\begin{array}{cc}
\lambda_{n-s+i}(\Tilde{A}+\Tilde{E})\leq \lambda_i+k|\lambda_i|, & \lambda_i(\Tilde{A}+\Tilde{E}) \geq \lambda_i -k|\lambda_i|
\end{array}
\end{equation*}\end{corollary}
\begin{proof}
We first apply Corollary 2.1.1 to $\Tilde{A}$ and $\Tilde{E}$ and obtain (14) for $i\in \{1,..,s\}$ with $\Tilde{k}=\|(\Tilde{A}^{1/2})^+\Tilde{E}(\Tilde{A}^{1/2})^+\|$. 
 Since $\Tilde{A}$ is positive semi-definite, its polar factor is itself. $A=D^{-1}\Tilde{A}D^{-*}$ is also positive semi-definite so its polar factor is itself too. Since $\Tilde{A}=DAD^*$, we can apply Theorem 3.1 to replace $\Tilde{k}$ by $k$ and get the result. 
 \end{proof}
In Theorem 3.1, when $A$ has full rank, the condition of $D^*D$ commuting with $V_rV_r^*$ becomes $D^*D$ commuting with $VV^*=I$ which holds for any $D$. Hence this extra constraint on $D$ compared to [6, Corollary 3.4] and [8, Corollary 2.2] vanishes so $D$ can be any invertible matrix. Hence, Corollary 3.1.1 recovers [8, Corollary 2.2] when $A$ is positive definite. When $r<n$, the constraint is active and $D$ must be such that $D^*D$ commute with $V_rV_r^*$. By [7, Theorem 2.5.5], this implies that $D^*D$ is diagonisable through $V$ as $D^*D=V\Lambda V^*$. Using [7, Theorem 7.3.2], if $D=U\Sigma W^*$ as its SVD, $W=VS$ where $S$ is some permutation matrix such that $S\Sigma S^*=\Lambda$.  

\subsection{Comparison to Weyl's absolute bound}
We have derived a relative perturbation bound giving (1) and (2) in Theorem 2.1. It is natural to ask under what circumstance the relative bound can be sharper than the Weyl's absolute  bound. If we first assume that the Hermitian matrix $A$ has full rank, then Theorem 2.1 coincides with the main theorem in [4], which is for $E$ being Hermitian, $i\in\{1,...,n\}$ ,  and $k=\|A^{-1/2}EA^{-1/2}\|\leq 1$, \begin{equation}
  |\lambda_i(A+E)-\lambda_i(A)|\leq k|\lambda_i(A)|
\end{equation}
where the orderings of $\lambda_i(A)$ and $ \lambda_i(A+E)$ are decreasing as before. On the other hand, the Weyl's absolute bound under the same context is\begin{equation}
    |\lambda_i(A+E)-\lambda_i(A)|\leq \|E\|
     \end{equation} 
To have (26) being sharper than (27) for the same index $i$ is equivalent to have $k|\lambda_i(A)|\leq \|E\|$, which will depend on the specific $i$ of interest. We will try to give a sufficient condition for it using the Ostrowski's bound [5, Theorem 2.1].

\begin{theorem}
    
Let $A$ be Hermitian non-singular with polar factor $P$, and $j'$ be any index such that $|\lambda_{j'}(P^{-1/2}EP^{-1/2})|=\max\limits_j|\lambda_j(P^{-1/2}EP^{-1/2})|$. Let $E$ be Hermitian. If for any fixed $i\in \{1,...,n\}$, \begin{equation}
    \frac{|\lambda_i(A)|}{\min\limits_l|\lambda_l(A)|}|\lambda_{j'}(E)|\leq \|E\|
\end{equation}
then we have $k|\lambda_i(A)|\leq \|E\|$ where $k=\|A^{-1/2}EA^{-1/2}\|$. \end{theorem}
\begin{proof}
 Let $A=VDV^*$ be its spectral decomposition. Then $P=V|D|V^*$ and $P^{-1/2}=V|D|^{-1/2}V^*$, both of which are Hermitian. As $E$ is Hermitian, so is $P^{-1/2}EP^{-1/2}$. Therefore, $|\lambda_{j'}(P^{-1/2}EP^{-1/2})|=\max\limits_j|\lambda_j(P^{-1/2}EP^{-1/2})|=\|P^{-1/2}EP^{-1/2}\|$. We can apply [5, Theorem 2.1] to $P^{-1/2}EP^{-1/2}$ and $E=P^{1/2}(P^{-1/2}EP^{-1/2})(P^{1/2})^*$ with the index being $j'$ together with $\|P^{-1}\|=\|A^{-1}\|=1/\min\limits_l|\lambda_l(A)|$ to obtain that
\begin{equation}
\begin{aligned}
    |\lambda_{j'}(E)|\geq \frac{|\lambda_{j'}(P^{-1/2}EP^{-1/2})|}{\|(P^{1/2}P^{1/2})^{-1}\|}=\frac{\|P^{-1/2}EP^{-1/2}\|}{\|P^{-1}\|}
    \overset{(4)}{=}k\min\limits_l|\lambda_l(A)|
    \end{aligned}
\end{equation}
 We substitute (29) into (28) which is given, and obtain that \begin{equation*}
\begin{aligned}
    \|E\|\overset{(28)}{\geq} |\lambda_{j'}(E)|\frac{|\lambda_i(A)|}{\min\limits_l|\lambda_l(A)|}
    \overset{(29)}{\geq} k\min\limits_l|\lambda_l(A)|\frac{|\lambda_i(A)|}{\min\limits_l|\lambda_l(A)|}=k|\lambda_i(A)|
     \end{aligned}
     \end{equation*}\end{proof}  
 
In general, it may depend on the index $i\in \{1,...,n\}$ of interest to have either $ k|\lambda_i(A)|\leq \|E\|$ or the contrary, but Theorem 3.2 ensures that there always exists some index $i\in \{1,...,n\}$  such that the relative bound in (26) is sharper than the Weyl's bound in (27) as illustrated in the following result:

\begin{theorem}
    Let $A$ be Hermitian non-singular and $E$ be any Hermitian matrix, if $k=\|A^{-1/2}EA^{-1/2}\|$, then there exists some index $i\in \{1,...,n\}$ such that \begin{equation*}
        |\lambda_i(A+E)-\lambda_i(A)|\leq k|\lambda_i(A)|\leq\|E\|
    \end{equation*}
\end{theorem}
\begin{proof}
    Take $i\in\{1,...,n\}$ such that $|\lambda_i(A)|=\min\limits_l|\lambda_l(A)|$. Since  $\|E\|=\max\limits_j|\lambda_j(E)|$ because $E$ is Hermitian, for any $j'\in \{1,...,n\}$, \begin{equation*}
         \frac{|\lambda_i(A)|}{\min\limits_l|\lambda_l(A)|}|\lambda_{j'}(E)|=|\lambda_{j'}(E)|\leq \|E\|
    \end{equation*}
So (28) is satisfied. By Theorem 3.2, we have $|\lambda_i(A+E)-\lambda_i(A)|\leq k|\lambda_i(A)|\leq \|E\|$.

\end{proof}

More generally when $A$ has rank $r\leq n$ in Theorem 2.1, we want to compare \begin{equation}
    \lambda_{n-r+i}(A+E)\leq \lambda_i(A)+k|\lambda_i(A)|
\end{equation}
 with its corresponding Weyl's bound $\lambda_{n-r+i}(A+E)\leq \lambda_{n-r+i}(A)+\|E\|$,  and \begin{equation}
\lambda_i(A+E)\geq \lambda_i(A)-k|\lambda_i(A)|
 \end{equation} with its corresponding Weyl's bound $\lambda_i(A+E)\geq \lambda_i(A)-\|E\|$ for $i\in \{1,...,r\}$, when $A$ is potentially singular and $k=\|(A^{1/2})^+E(A^{1/2})^+\|$. 
 We can give a sufficient condition for the relative bounds in (30) and (31) being sharper than the Weyl's bounds, though it is less likely to be met as $\mbox{rank}(A)=r$ decreases away from $n$.   \begin{theorem}
    Let $A$ be Hermitian with rank $r\leq n$ and polar factor $P$. Let $j'$ be any index that $|\lambda_{j'}((P^{1/2})^+E(P^{1/2})^+)|=\max\limits_{j }|\lambda_j((P^{1/2})^+E(P^{1/2})^+)|$ where the eigenvalues are in the default ordering on page 2, and $k=\|(A^{1/2})^+E(A^{1/2})^+\|$. If for
     $i\in\{1,...,r\}$, \begin{equation}
\lambda_i(A)-\lambda_{n-r+i}(A)+\frac{\max \{|\lambda_{j'}(E)|, |\lambda_{n-r+j'}(E)| \}}{\min\limits_{l: \lambda_l(A)\neq 0}|\lambda_l(A)|}|\lambda_i(A)|\leq \|E\|
 \end{equation} 
 then $\lambda_{i}(A)+k|\lambda_i(A)|\leq \lambda_{n-r+i}(A)+\|E\|$ and $\lambda_{i}(A)-k|\lambda_i(A)|\geq \lambda_{i}(A)-\|E\|$. \end{theorem}
\begin{proof}
Let $A=V_rD_rV_r^*$ be its thin spectral decomposition, then $P=V_r|D_r|V_r^*$, and $(P^{1/2})^+=(V_r|D_r|^{1/2}V_r^*)^+=V_r|D_r|^{-1/2}V_r^*$. Since $(P^{1/2})^+E(P^{1/2})^+$ is Hermitian because $(P^{1/2})^+$ and $E$ are Hermitian, we have\begin{equation*}
   k=\|(A^{1/2})^+E(A^{1/2})^+\|\overset{(4)}{=}\|(P^{1/2})^+E(P^{1/2})^+\|=|\lambda_{j'}((P^{1/2})^+E(P^{1/2})^+)|  
   \end{equation*}
   Since $\mbox{rank}((P^{1/2})^+E(P^{1/2})^+)\leq \mbox{rank}((P^{1/2})^+)=r$,  the index $j'\in \{1,...,r\}$, because $\lambda_{j}((P^{1/2})^+E(P^{1/2})^+)=0$ for any $j>r$, assuming $(P^{1/2})^+E(P^{1/2})^+\neq \mathbf{0}$.\footnote{If $(P^{1/2})^+E(P^{1/2})^+=V_r|D_r|^{-1/2}V_r^*EV_r|D_r|^{-1/2}V_r^*=\mathbf{0}$, multiplying both sides by $|D_r|^{1/2}V_r^*$ on the left and $V_r|D_r|^{1/2}$ on the right we get $V_r^*EV_r=\mathbf{0}$. We can still take $j'\in \{1,...,r\}$ since all the eigenvalues are zero and the rest of the deviation will hold true. } Then for such $j'\in \{1,...,r\}$, using $\mbox{eig}(XY)=\mbox{eig}(YX)$ [7, Theorem 1.3.22],
\begin{equation}
\begin{aligned}
   k= \|(P^{1/2})^+E(P^{1/2})^+\|&=|\lambda_{j'}((P^{1/2})^+E(P^{1/2})^+)|\\
    &=|\lambda_{j'}(\underbrace{V_r}_{X}\underbrace{|D_r|^{-1/2}V_r^*EV_r|D_r|^{-1/2}V_r^*}_{Y})| \\
    &=|\lambda_{j'}(|D_r|^{-1/2}V_r^*EV_r|D_r|^{-1/2}V_r^*V_r)| \\ &=|\lambda_{j'}(|D_r|^{-1/2}V_r^*EV_r|D_r|^{-1/2})|  \\
        \end{aligned}
\end{equation}
On the other hand, since $V_r^*EV_r=|D_r
|^{1/2}(|D_r|^{-1/2}V_r^*EV_r|D_r|^{-1/2})|D_r|^{1/2}$ where $|D_r|^{1/2}$ is Hermitian non-singular, we can apply [5, Theorem 2.1] to $V_r^*EV_r$  and $|D_r|^{-1/2}V_r^*EV_r|D_r|^{-1/2}$ at index $j'\in \{1,...,r\}$, obtaining that
\begin{equation}
       |\lambda_{j'}(V_r^*EV_r)|\geq \frac{|\lambda_{j'}(|D_r|^{-1/2}V_r^*EV_r|D_r|^{-1/2})|}{\|(|D_r|^{1/2}|D_r|^{1/2})^{-1}\|}
       =\frac{|\lambda_{j'}(|D_r|^{-1/2}V_r^*EV_r|D_r|^{-1/2})|}{\||D_r|^{-1}\|}
       \end{equation}
We substitute (33) and that $\||D_r|^{-1}\|=1/\min\limits_{l: \lambda_l(A)\neq 0}|\lambda_l(A)|$ into (34), and get \begin{equation} 
    |\lambda_{j'}(V_r^*EV_r)|\geq \|(P^{1/2})^+E(P^{1/2})^+\|\min\limits_{l: \lambda_l(A)\neq 0}|\lambda_l(A)|=k\min\limits_{l: \lambda_l(A)\neq 0}|\lambda_l(A)|
    \end{equation}
By [7, Theorem 4.3.28], and (7) but with $E$, since $j'\in \{1,...,r\}$, \begin{equation*}
\lambda_{n-r+j'}(E)=\lambda_{n-r+j'}(V^*EV)\leq \lambda_{j'}(V_r^*EV_r)\leq\lambda_{j'}(V^*EV)=\lambda_{j'}(E)
   \end{equation*} which means that $|\lambda_{j'}(V_r^*EV_r)|\leq \max\{|\lambda_{j'}(E)|, |\lambda_{n-r+j'}(E)| \} $. Then we substitute this and (35) into (32) which is given and get\begin{equation*}
   \begin{aligned}
       \|E\|&\overset{(32)}{\geq }\lambda_i(A)-\lambda_{n-r+i}(A)+\frac{\max \{|\lambda_{j'}(E)|, |\lambda_{n-r+j'}(E)| \}}{\min\limits_{l: \lambda_l(A)\neq 0}|\lambda_l(A)|}|\lambda_i(A)| \\
       &\geq\lambda_i(A)-\lambda_{n-r+i}(A)+\frac{|\lambda_{j'}(V_r^*EV_r)|}{\min\limits_{l: \lambda_l(A)\neq 0}|\lambda_l(A)|}|\lambda_i(A)|  \\
       &\overset{(35)}{\geq}\lambda_i(A)-\lambda_{n-r+i}(A)+k|\lambda_i(A)|        \end{aligned}      
       \end{equation*}
       which means that $\lambda_{i}(A)+k|\lambda_i(A)|\leq \lambda_{n-r+i}(A)+\|E\|$. Also, since $\lambda_{n-r+i}(A)\leq \lambda_i(A)$, we have $k|\lambda_i(A)|\leq \|E\|$ so $\lambda_i(A)-k|\lambda_i(A)|\geq \lambda_i(A)-\|E\|$. As a result, the relative bounds in (30) and (31) are sharper than their corresponding Weyl's bounds. \end{proof}
   
     Theorem 3.4 generalises Theorem 3.2 to $A$ having rank $r\leq n$. We will illustrate when Theorem 3.4 could be useful in the next result which is an extension of Theorem 3.3, and it also provides justification that in some cases we should consider to use Theorem 2.1 rather than the Weyl's absolute bound. 
    
       \begin{corollary}
           Let $A$ be Hermitian of rank $r\leq n$ with eigenvalues $\lambda_1\geq...\geq \lambda_r, \lambda_{r+1}=...=\lambda_n=0$. Let $E$ be Hermitian, and $k=\|(A^{1/2})^+E(A^{1/2})^+\|$. If the eigenvalue with the smallest non-zero absolute value of $A$ has at least $n-r+1$ algebraic multiplicity, then there exists some $i\in \{1,...,r\}$ such that 
\begin{equation*}
\begin{array}{cc}
\lambda_{i}+k|\lambda_i|\leq \lambda_{n-r+i}+\|E\|, &  \lambda_i-k\|\lambda_i\|\geq \lambda_i-\|E\|\\
\end{array}
    \end{equation*} 
       \end{corollary}
       \begin{proof}
           Take $i=\min\{j\in \{1,...,r\}: |\lambda_j|=\min\limits_{l:\lambda_l\neq 0}|\lambda_l|\}$. Since there are at least $n-r+1$ copies of $\lambda_i$ among the eigenvalues of $A$, by the ordering specified, $\lambda_{n-r+i}=\lambda_i$. Also, $|\lambda_i|=\min\limits_{l:\lambda_l\neq 0}|\lambda_l|$ by the definition of $i$. Hence, for any $j'\in \{1,...,n\}$,
        \begin{equation*}
        \begin{aligned}
\lambda_i-\lambda_{n-r+i}+\frac{\max \{|\lambda_{j'}(E)|, |\lambda_{n-r+j'}(E)| \}}{\min\limits_{l: \lambda_l\neq 0}|\lambda_l|}|\lambda_i|&=\max \{|\lambda_{j'}(E)|, |\lambda_{n-r+j'}(E)|\} \\ 
&\leq \max\limits_{l'}|\lambda_{l'}(E)|=\|E\|\end{aligned}
        \end{equation*}
so (32) is satisfied for such $i\in \{1,...,r\}$. By Theorem 3.4, we have $\lambda_{i}+k|\lambda_i|\leq \lambda_{n-r+i}+\|E\|$ and $\lambda_i-k|\lambda_i|\geq \lambda_i-\|E\|$. \end{proof}

\section{Relative bounds for singular values}
 Theorem 2.1 can be utilised to derive a relative bound on singular values of general matrices using the strategy mentioned in [5] as illustrated in the following result. It is structurally similar to [4, Theorem 3.1], but also applicable to singular values of any matrices, because Theorem 2.1 can be applied to Hermitian matrices without full rank. We will also see that the relative bound for singular values below involves a constant $k=\|S^+ES^+\|$ where $S$ can be thought of as a pseudo $A^{1/2}$, resembling Theorem 2.1 for eigenvalues. We first assume $m\geq n$ for definiteness.
 
 \begin{theorem}
Let $A$ be an $m\times n$ matrix of rank $r$ with SVD $A=U\Sigma V^*=U_r\Sigma_r V_r^*$. Let $E$ be any $m\times n$ matrix and $S^+=V_r\Sigma_r^{-1/2} U_r^*$. If $k=\|S^+ES^+\|\leq 1$, then \begin{equation}
    \sigma_{m+n-2r+i}(A+E)\leq \sigma_i(A)+k\sigma_i(A)
\end{equation} for any $i\in \{1,...,2r-m\}$, and also\begin{equation}
    \sigma_i(A+E)\geq \sigma_i(A)-k\sigma_i(A) 
\end{equation} for any $i\in \{1,...,r\}$.
\end{theorem}
\begin{proof}
  Let $\Tilde{A}=\left[\begin{array}{cc}
     & A \\
   A^*  & 
\end{array}\right]$, $\Tilde{E}=\left[\begin{array}{cc}
     & E \\
  E^*   & 
\end{array}\right]$, $\Tilde{A}+\Tilde{E}=\left[\begin{array}{cc}
     &  A+E\\
   (A+E)^*  & 
\end{array}\right]$ which are all Hermitian $(m+n)\times (m+n)$ matrices. We first want to show (36). Using Theorem 7.3.3 in [7], \begin{equation*}
\Tilde{A}=\left[\begin{array}{cc}
         & A \\
      A^*   & 
    \end{array}\right]=\left[\begin{array}{ccc}
       \frac{1}{\sqrt{2}}U  &  -\frac{1}{\sqrt{2}}U & U^\perp\\
  \frac{1}{\sqrt{2}}V &   \frac{1}{\sqrt{2}}V & \mathbf{0}    \end{array}\right] \left[\begin{array}{ccc}
     \Sigma  &  & \\
       & -\Sigma & \\
     &  & \mathbf{0}
  \end{array}\right]\left[\begin{array}{ccc}
\frac{1}{\sqrt{2}}U^* & \frac{1}{\sqrt{2}}V^*\\
-\frac{1}{\sqrt{2}}U^* & \frac{1}{\sqrt{2}}V^*  \\
{U^\perp}^* & \mathbf{0} \\ \end{array}\right]
\end{equation*} is its spectral decomposition with eigenvalues $\sigma_1(A),..., \sigma_n(A), -\sigma_1(A),...,-\sigma_n(A)$ and $m-n$ copies of zeros. Since $A$ has rank $r$, $\Tilde{A}$ has rank $2r$. Our aim is to apply Theorem 2.1 to $\Tilde{A}$ and $\tilde{E}$ suggesting that if $\|(\Tilde{A}^{1/2})^+\Tilde{E}(\Tilde{A}^{1/2})^+\|\leq 1$, then for any $i\in \{1,...,2r\},$ \begin{equation}
\begin{aligned}
    \lambda_{m+n-2r+i}(\Tilde{A}+\Tilde{E})&\leq \lambda_i(\Tilde{A})+\|(\Tilde{A}^{1/2})^+\Tilde{E}(\Tilde{A}^{1/2})^+\||\lambda_i(\Tilde{A})|\\
   \end{aligned}
\end{equation}
Let $\Tilde{P}$ be the polar factor of $\Tilde{A}$. By the definition of normal square root, we get \begin{equation*}
\begin{aligned}
    \Tilde{P}^{1/2}&=\left[\begin{array}{ccc}
       \frac{1}{\sqrt{2}}U  &  -\frac{1}{\sqrt{2}}U & U^\perp\\
  \frac{1}{\sqrt{2}}V &   \frac{1}{\sqrt{2}}V & \mathbf{0}    \end{array}\right] \left[\begin{array}{ccc}
     \Sigma^{1/2}  &  & \\
       & \Sigma^{1/2} & \\
     &  & \mathbf{0}
  \end{array}\right]\left[\begin{array}{ccc}
\frac{1}{\sqrt{2}}U^* & \frac{1}{\sqrt{2}}V^*\\
-\frac{1}{\sqrt{2}}U^* & \frac{1}{\sqrt{2}}V^*  \\
{U^\perp}^* & \mathbf{0} \\ \end{array}\right] \\
&=\left[\begin{array}{ccc}
       \frac{1}{\sqrt{2}}U\Sigma^{1/2}  &  -\frac{1}{\sqrt{2}}U\Sigma^{1/2} & \mathbf{0}\\
  \frac{1}{\sqrt{2}}V\Sigma^{1/2} &   \frac{1}{\sqrt{2}}V\Sigma^{1/2} & \mathbf{0}    \end{array}\right]\left[\begin{array}{ccc}
\frac{1}{\sqrt{2}}U^* & \frac{1}{\sqrt{2}}V^*\\
-\frac{1}{\sqrt{2}}U^* & \frac{1}{\sqrt{2}}V^*  \\
{U^\perp}^* & \mathbf{0} \\ \end{array}\right]\\
&=\left[\begin{array}{cc}
   U\Sigma^{1/2}U^*  &  \\
     & V\Sigma^{1/2}V^*
\end{array}\right]=\left[\begin{array}{cc}
   U_r\Sigma_r^{1/2}U_r^*  &  \\
     & V_r\Sigma_r^{1/2}V_r^*
\end{array}\right]\end{aligned}
\end{equation*}
because there are $r$ non-zero diagonal entries in $\Sigma^{1/2}$. Since $\Tilde{P}^{1/2}$ is block diagonal, \begin{equation*}
  (\Tilde{P}^{1/2})^+ =\left[\begin{array}{cc}
   (U_r\Sigma_r^{1/2}U_r^* )^+ &  \\
     & (V_r\Sigma_r^{1/2}V_r^*)^+
\end{array}\right] =\left[\begin{array}{cc}
   U_r\Sigma_r^{-1/2}U_r^*  &  \\
     & V_r\Sigma_r^{-1/2}V_r^*
\end{array}\right]\end{equation*}
Substituting this into $(\Tilde{P}^{1/2})^+\Tilde{E}(\Tilde{P}^{1/2})^+$, we obtain that
\begin{equation*}\begin{aligned}
   (\Tilde{P}^{1/2})^+\Tilde{E}(\Tilde{P}^{1/2})^+
&=\left[\begin{array}{cc}
    &U_r\Sigma_r^{-1/2}U_r^* E\\
V_r\Sigma_r^{-1/2}V_r^*E^* & 
\end{array}\right]\left[\begin{array}{cc}
   U_r\Sigma_r^{-1/2}U_r^*  &  \\
     & V_r\Sigma_r^{-1/2}V_r^*
\end{array}\right]\\
&=\left[\begin{array}{cc}
    &U_r\Sigma_r^{-1/2}U_r^* E V_r\Sigma_r^{-1/2}V_r^*\\
(U_r\Sigma_r^{-1/2}U_r^* E V_r\Sigma_r^{-1/2}V_r^*)^*& 
\end{array}\right]
\end{aligned}
\end{equation*}
    which is the Jordan-Wielandt matrix of $U_r\Sigma_r^{-1/2}U_r^* E V_r\Sigma_r^{-1/2}V_r^*$ so its eigenvalues are the singular values of $U_r\Sigma_r^{-1/2}U_r^* E V_r\Sigma_r^{-1/2}V_r^*$, their negations and zeros. As $U$ is $m\times n$ orthonormal and $V$ is $n\times n$ unitary, $\|UXV^*\|=\|X\|=\|X^*\|=\|UX^*V^*\|=\|(UX^*V^*)^*\|=\|VXU^*\|$ for any $X$.  Therefore,
 \begin{equation}
       \begin{aligned}
            \|(\Tilde{P}^{1/2})^+\Tilde{E}(\Tilde{P}^{1/2})^+\|&=\max\limits_i|\lambda_i((\Tilde{P}^{1/2})^+\Tilde{E}(\Tilde{P}^{1/2})^+)|\\
            &=\sigma_1(U_r\Sigma_r^{-1/2}U_r^* E V_r\Sigma_r^{-1/2}V_r^*) \\            &=\|U_r\Sigma_r^{-1/2}U_r^* E V_r\Sigma_r^{-1/2}V_r^* \| \\
            &=\|U\underbrace{\left[\begin{array}{cc}
\Sigma_r^{-1/2}&  \\
                 & \mathbf{0}
            \end{array}\right]U^* E V\left[\begin{array}{cc}
\Sigma_r^{-1/2} &  \\
                 & \mathbf{0}
            \end{array}\right]}_{X}V^* \|      \\
            &=\|\left[\begin{array}{cc}
\Sigma_r^{-1/2}&  \\
                 & \mathbf{0}
            \end{array}\right]U^* E V\left[\begin{array}{cc}
\Sigma_r^{-1/2} &  \\
                 & \mathbf{0}
            \end{array}\right] \|   \\
            &=\|V\left[\begin{array}{cc}
\Sigma_r^{-1/2}&  \\
                 & \mathbf{0}
            \end{array}\right]U^* E V\left[\begin{array}{cc}
\Sigma_r^{-1/2} &  \\
                 & \mathbf{0}
            \end{array}\right]U^*\|  \\
            &=\|V_r\Sigma_r^{-1/2} U_r^* E V_r\Sigma_r^{-1/2}U_r^* \| =\|S^+ES^+\|       
            \end{aligned}   \end{equation}
Since $\|(\Tilde{A}^{1/2})^+\Tilde{E}(\Tilde{A}^{1/2})^+\|\overset{(4)}{=}\|(\Tilde{P}^{1/2})^+\Tilde{E}(\Tilde{P}^{1/2})^+\|\leq 1$, by (1) in Theorem 2.1, we have (38) for $i\in \{1,...,2r\}$. Substituting (39) into (38), we get\begin{equation}
     \lambda_{m+n-2r+i}(\Tilde{A}+\Tilde{E})\leq \lambda_i(\Tilde{A})+k|\lambda_i(\Tilde{A})|
     \end{equation}    where $\lambda_1(\Tilde{A})\geq...\geq \lambda_{2r}(\Tilde{A}), \lambda_{2r+1}(\Tilde{A})=...=\lambda_{m+n}(\Tilde{A})=0$ and the eigenvalues of $\Tilde{A}+\Tilde{E}$ are in decreasing order, as in Theorem 2.1. Since the eigenvalues of $\Tilde{A}+\Tilde{E}$ are $\sigma_1(A+E),...,\sigma_n(A+E), 0,..,0, -\sigma_n(A+E),...,-\sigma_1(A+E)$, for $m+n-2r+i\leq n$ which is $i\leq 2r-m$, $\lambda_{m+n-2r+i}(\Tilde{A}+\Tilde{E})=\sigma_{m+n-2r+i}(A+E)$. The eigenvalues of $\Tilde{A}$ in the ordering specified before are $\sigma_1(A),...,\sigma_r(A), -\sigma_r(A),...,-\sigma_1(A), 0,...,0$ so for $i\leq r$, we have $\lambda_i(\Tilde{A})=\sigma_i(A)$. Since $r\leq m$, $2r-m\leq r$, so for $i\leq 2r-m$, we have $m+n-2r+i\leq n$ and $i\leq r$, giving $\lambda_{m+n-2r+i}(\Tilde{A}+\Tilde{E})=\sigma_{m+n-2r+i}(A+E)$ and $\lambda_i(\Tilde{A})=\sigma_i(A)$. Substituting them into (40), we get  (36) for $i\in \{1,...,2r-m\}.$

To show (37), we apply (2) in Theorem 2.1 to $\Tilde{A}$ and $\Tilde{E}$, and obtain for $i\in\{1,...,2r\}$ \begin{equation}
     \lambda_i(\Tilde{A}+\Tilde{E})\geq \lambda_i(\Tilde{A})-k|\lambda_i(\Tilde{A})|
 \end{equation}
  When $i\leq r$, $\lambda_i(\Tilde{A})=\sigma_i(A)$ and when $i\leq n$, $\lambda_i(\Tilde{A}+\Tilde{E})=\sigma_i(A+E)$, so when $i\leq r$ we have both. Substituting them into (41), we get (37) for $i\in \{1,...,r\}$. 
 \end{proof}
 \begin{remark}
     \textnormal{Since $S^+=V_r\Sigma_r^{-1/2}U_r^*$, $S=U_r \Sigma_r^{1/2}V_r^*$. Then
$S^*S=V_r\Sigma_rV_r^*$ which is the right polar factor of the tall matrix $A$ by [7, Theorem 7.3.1]. In the case that $A$ is normal, $A^{1/2}=VD^{1/2}V^*$ and $(A^{1/2})^*A^{1/2}=V(D^{1/2})^*D^{1/2}V^*=V|D|V^*$ which is the polar factor of $A$. Therefore, such $S$ can be regarded as pseudo $A^{1/2}$  because it exhibits a similar property of  $A^{1/2}$ when $A$ is normal. Hence, $k=\|S^+ES^+\|$ in Theorem 4.1 resembles $k=\|(A^{1/2})^+E(A^{1/2})^+\|$ used in Theorem 2.1.}
 \end{remark}

 When $m<n$, we can apply Theorem 4.1 to $A^*$, $E^*$ with $k=\|\Tilde{S}^+E^*\Tilde{S}^+\|=\|S^+ES^+\|$ where $\Tilde{S}^+=U_r\Sigma_r^{-1/2}V_r^*=(S^+)^*$. Since singular values are invariant under matrix Hermitian, we will get (36) for $i\in \{1,..., 2r-n\}$ and (37) for $i \in \{1,...,2r\}$. Therefore, in general for any $m \times n$ matrices $A$ and $E$ in the context of Theorem 4.1, we have (36) for $i \in \{1,..., 2r-\max\{m, n\}\}$ and (37) for $i \in \{1,..., 2r\}$, where $r=\mbox{rank}(A)$.

When $m=n$, $A$ is a square matrix of rank $r$ and has a left and right polar decomposition as $A=P_1Q=QP_2$. Let $A=U\Sigma V^*$ be its SVD. By [7, Theorem 7.3.1], $P_1=(AA^*)^{1/2}=U\Sigma U^*$, and $P_2=(A^*A)^{1/2}=V\Sigma V^*$. Therefore, $P_1^{1/2}=U\Sigma^{1/2}U^*$ and $P_2=V\Sigma^{1/2}V^*$ which leads to $(P_1^{1/2})^+=U_r\Sigma_r^{-1/2} U_r^*$ and $(P_2^{1/2})^+=V_r\Sigma_r^{-1/2} V_r^*$. Recall the third and seventh line in (39), we have \begin{equation*}
\begin{aligned}
     \|(P_1^{1/2})^+E(P_2^{1/2})^+\|
    &=\|U_r\Sigma_r^{-1/2}U_r^* E V_r\Sigma_r^{-1/2}V_r^* \|\\
    &\overset{(39)}{=}\|V_r\Sigma_r^{-1/2} U_r^* E V_r\Sigma_r^{-1/2}U_r^* \| \\ 
    &=\|S^+ES^+\|=k
    \end{aligned}
\end{equation*} Hence, we can replace $k=\|S^+ES^+\|$ by $\|(P_1^{1/2})^+E(P_2^{1/2})^+\|$ in Theorem 4.1, giving an analogue of [4, Theorem 3.1] for $A$ being any general square matrix. \\

\textbf{Declaration of interests:} none

\textbf{Acknowledgements.}  I would like to thank Professor Y. Nakatsukasa who provided me with clear guidance and helpful advice for this paper.
\section*{References}
[1] F.L. Bauer, C.T. Fike, \textit{Norms and exclusion theorems}, Numer. Math. 2 (1960) 137-141. \\\relax
[2] R. Bouldin, \textit{The pseudo-inverse of a product}, SIAM J. Appl. Math. 24 (1973) 489-495. \\\relax
[3] J. Demmel, K. Veselić, \textit{Jacobi’s method is more accurate than QR}, SIAM J. Matrix Anal. Appl. 13
(1992) 1204–1245. \\\relax
[4] F.M. Dopico, J. Moro, J.M. Morela, \textit{Weyl-type relative perturbation bounds for eigensystems of Hermitian matrices}, Linear Algebra Appl. 309 (2000) 3-18.\\\relax
[5] S.C. Eisenstat, I.C.F. Ipsen, \textit{Relative perturbation techniques for singular value problems}, SIAM
J. Numer. Anal. 32 (1995) 1972-1988.\\\relax 
[6] S.C. Eisenstat, I.C.F. Ipsen, \textit{Three absolute perturbation bounds for matrix eigenvalues imply relative bounds}, SIAM J. Matrix Anal. Appl. 20 (1998) 149-158. \\\relax
[7] R.A. Horn, C.R. Johnson, \textit{Matrix Analysis}, second ed., Cambridge University Press, Cambridge, 2013. \\\relax
[8] I.C.F. Ipsen, \textit{Relative perturbation results for matrix eigenvalues and singular values}, Acta Numer. 7 (1998) 151-201.\\\relax
[9] S. Izumino,  \textit{The product of operators with closed range and an extension of the reverse order law}, Tôhoku Math. J. 34 (1982) 43-52.\\\relax
[10] R. Mathias, \textit{Spectral perturbation bounds for positive definite matrices}, SIAM J. Matrix Anal. Appl. 18 (1997) 959-980.\\\relax
[11] K. Veselić, I. Slapničar, \textit{Floating-point perturbations of Hermitian matrices}, Linear
Algebra Appl. 195 (1993) 81-116. \\\relax

\end{titlepage}